\documentclass[11pt]{amsart}
\usepackage[margin=1.12in]{geometry}
\usepackage{amsmath,amssymb,amsthm,mathtools,booktabs,array,longtable}
\usepackage{microtype}
\usepackage{xurl}
\usepackage[hidelinks]{hyperref}
\hypersetup{pdftitle={An improved result on Chern's conjecture}}
\numberwithin{equation}{section}
\newtheorem{theorem}{Theorem}[section]

\newtheorem{lemma}[theorem]{Lemma}

\theoremstyle{remark}

\newcommand{\tr}{\operatorname{tr}}

\allowdisplaybreaks[2]
\newcommand{\one}{\mathbf 1}
\DeclareMathOperator{\diag}{diag}
\title[On Chern's conjecture]{An improved result on Chern conjecture*}
\author [Q.-M. Cheng, F. Li, G. Wei and J. Yin]{Qing-Ming Cheng, Fengjiang Li,  Guoxin Wei and Jiabin Yin}
\address{Qing-Ming Cheng \newline
\indent
Mathematical Science Research Center, \newline 
\indent Chongqing University of Technology, Chongqing 400054, P. R. China. 
\vskip 1mm
\indent
Osaka Center Advanced Mathematical Institute \newline
\indent
Osaka Metropolitan University, Osaka 558-8585, Japan. \newline
\indent qingmingcheng@yahoo.com, chengqingming@cqut.edu.cn }
\address{Fengjiang Li \newline
\indent
Mathematical Science Research Center, \newline 
\indent Chongqing University of Technology, Chongqing 400054, P. R. China. \newline 
\indent fengjiangli@cqut.edu.cn }

\address{Guoxin Wei \newline
\indent   School of Mathematical Sciences, \newline 
\indent  South China Normal University, Guangzhou 510631, P. R. China. \newline 
\indent weiguoxin@tsinghua.org.cn}

\address{Jiabin Yin \newline
\indent 	School of Mathematics and Statistics, \newline
\indent 	Xinyang Normal University, Xinyang 464000,
	People's Republic of China \newline
\indent jiabinyin@126.com}

\thanks{*This work was supported by  National Natural Science Foundation of China (Grant Nos.
12301062, 12571050, 12671065, 12201138), Natural Science Foundation of Chongqing
(No.CSTB2024NSCQ-MSX0537), Japan Society for the Promotion of Science Grant-in-Aid
for Scientific Research (C) (Grant No. 25K06992) and NSF of Henan Province (Grant No. 262300421869).} 
\date{}
\begin{document}
\begin{abstract}
Let $X: M^n\to \mathbb S^{n+1}(1)$, $n\ge4$, be a complete minimal hypersurface with constant scalar curvature. 
Chern conjecture says that  a closed minimal hypersurface with constant scalar curvature, $S\geq 2n$ if $S>n$,
where $S$ is the squared norm of the second fundamental form.
As a partial result, Peng-Terng  \cite{pt1, pt2}, Yang-Cheng \cite{yc1, yc2, yc3}
and Suh-Yang \cite{sy} obtained that if $S>n$, then $S>n+ \dfrac{3}7n$. Recently, Chen \cite{c} has proved  $S>n+\frac{26}{59}n$ 
if $S>n$. Since resolving Chern conjecture is a hard problem, as a midway for resolving Chern conjecture, one 
wants to prove $S>n+\frac12 n$ if $S>n$.  In this paper,  we give a positive solution for this midway problem. 
We should remark that our assumption is completeness, which is weaker than closedness in Chen \cite{c}.
\end{abstract}
\maketitle

\section{Introduction}
Let $X:M^n\to\mathbb S^{n+1}(1)$ be a  connected minimal hypersurface. We 
denote the shape operator  and the second fundamental form by $\mathbf A$ 
and $\vec h$ of $X:M^n\to\mathbb S^{n+1}(1)$, respectively.
In this paper, hypersurfaces that we consider are connected and immersed if there is  no an obvious mention. 
We define $S$ and $f_k$ by 
\[
 S=|\vec h|^2=\operatorname{tr}(\mathbf A^2),\qquad f_k=\operatorname{tr}(\mathbf A^k).
\]
According to the Gauss equations, it is well-known that, for a minimal  hypersurface $X:M^n\to\mathbb S^{n+1}(1)$, 
the scalar curvature is constant if and only if $S$ is constant. 
In view of the equation of Simons, it is easy to prove that for a closed minimal hypersurface $X:M^n\to\mathbb S^{n+1}(1)$  in the unit sphere,
$S\equiv 0 $ or $S\equiv n$ if $S\leq n$. Furthermore,  Chern-do Carmo-Kobayashi \cite{cdk} and Lawson \cite{l}
proved that only Clifford  torus  is the  minimal hypersurfaces $X:M^n\to\mathbb S^{n+1}(1)$  in the unit sphere with $S=n$.  
The following Chern 
conjecture is famous:

\vskip2mm
\noindent
{\bf Chern  conjecture}.
For $n$-dimensional closed  minimal  hypersurfaces $X:M^n\to\mathbb S^{n+1}(1)$  in the unit sphere
with constant scalar curvature,  if $S>n$, then  $S\geq 2n$.
\vskip2mm
\noindent
Many authors (see \cite{s, pt1, pt2, cheng96, yc1, yc2, yc3, sy, gt, ty, c}) made important contributions 
toward resolving the Chern conjecture.  
In particular,  Chen \cite{c} has recently proved  $S>n+\frac{26}{59}n$ if $S>n$. As a midway for resolving Chern conjecture, one 
wants to prove $S>n+\frac12 n$ if $S>n$. \par

Our main purpose in this paper is to realize this midway.  We obtain
the following:

\begin{theorem}\label{thm:main}
Let $X:M^n\to\mathbb S^{n+1}(1)$, $n\ge4$, be an $n$-dimensional  complete 
minimal hypersurface with constant scalar curvature.
If $S>n$, then $S> n+\frac{1}{2}n$.
\end{theorem}

Our proof uses the function introduced in the work of Yang and Cheng \cite{yc3},
with a fixed coefficient. By applying the generalized  maximum principle to this function, its 
gradient  gives an additional relation for the covariant derivative
of $\vec h$. We use this relation in both derivative estimates, 
which  is a new idea in resolving this kind of problem.
The tensor decompositions and their contractions follow the approach
of Yang and Cheng \cite{yc3}.
The method is also inspired by the recent work of Hu, Sun and Yin
\cite{hsy}. 
In Section~\ref{sec:preliminaries}, we give the notation, basic formulas
and lemmas needed in the proof. In Section 3, we prove main estimates for our proof. In Section 4,  we prove
the theorem~\ref{thm:main}.

\section{Notation and lemmas}\label{sec:preliminaries}

In this paper, we will use the notation,  symbols  and  the basic formulas in  Yang-Cheng \cite{yc3}.
Throughout Sections 2 and 3, we assume that $n\ge4$ and $S>n$, and put
\[
 t=\frac{S-n}{S}>0.
\]
At a fixed point, choose an orthonormal frame $\{e_1, \cdots, e_n\}$ such that components $h_{ij}$ of the second fundamental form $\vec h$
satisfy $h_{ij}=\lambda_i\delta_{ij}$, where $\lambda_i$'s denote the  principal curvatures.  
Let $h_{ijk}$ and $h_{ijkl}$ denote  the first and second covariant derivatives of the second
fundamental form. 
The Gauss equations, Codazzi equations and Ricci identities are
\begin{equation}
\begin{aligned}
& R_{ijkl}=\delta_{ik}\delta_{jl}-\delta_{il}\delta_{jk}
       +h_{ik}h_{jl}-h_{il}h_{jk},\\
     &  h_{ijk}=h_{ikj}, \\
 &h_{ijkl}-h_{ijlk}=\sum_a h_{ia}R_{ajkl}+\sum_a h_{aj}R_{aikl}.
\end{aligned}
\end{equation}
Since  $X:M^n\to\mathbb S^{n+1}(1)$ is minimal and  $S$ is constant, we have  
\[
 \sum_i\lambda_i=0,\qquad \sum_i h_{iik}=0,
 \qquad \sum_i\lambda_i h_{iik}=0.
\]
By making use of the notation and symbols $a_{ij}$ in \cite{yc3},  we obtain
$$
a_{ij}=(\lambda_i^2-\dfrac{f_3}{S}\lambda_i-\dfrac Sn)\delta_{ij}, \quad  b_{ij}=(\lambda_i^3-\dfrac75\dfrac{f_3}{S}\lambda_i^2)\delta_{ij}
\ \ \ \ i, j=1,2,\cdots, n
$$
and 
\[
\sum_ia_{ii}=0, \quad \sum_{i,j}a_{ij}h_{ij}=0, \quad  \sum_{i,j}a_{ij}^2=f_4-\frac{f_3^2}{S}-\frac{S^2}{n}=f.
\]
Letting
\[
{\bf 1}=(1, \cdots, 1)^T, \quad  {\vec\varphi}=\left(\frac{a_{11}}S,\cdots,\frac{a_{nn}}S\right)^T,
 \qquad {\vec q}=(b_{11},\cdots,b_{nn})^T,
\]
we denote by $P$ the orthogonal projection onto
$\operatorname{span}\{\one,(\lambda_1,\cdots,\lambda_n)^T\}^{\perp}$ and put
\begin{equation}\label{eq:projection-definition}
 \vec\psi=P\vec q,\qquad
 \vec\varphi_*=
 \begin{cases}
 \displaystyle\vec\varphi-
 \frac{\langle\vec\varphi,\vec\psi\rangle}{|\vec\psi|^2}\vec\psi,
       &\vec\psi\ne0,\\[5pt]
 \vec\varphi,&\vec\psi=0.
 \end{cases}
\end{equation}
Since $\vec\varphi\perp\one,(\lambda_1,\ldots,\lambda_n)^T$, the vector
$\vec\varphi_*$ is the orthogonal projection of $\vec\varphi$ onto
\[
 \operatorname{span}\{\one,(\lambda_1,\ldots,\lambda_n)^T,\vec q\}^{\perp}.
\]
We have
\begin{equation}\label{eq:2.2}
 |\vec\varphi|^2=\frac{f}{S^2},\qquad
 0\le |\vec\varphi_*|^2\le |\vec\varphi|^2.
\end{equation}

\begin{lemma} \label{lem:2.1} For a minimal hypersurface in the unit sphere with $n\ge4$ and $S>0$, we have 
\begin{equation}\label{eq:2.3}
2Sf+\dfrac{n-1}{n-2}f^2_3\leq (1-\dfrac2n)S^3.
\end{equation}
If
\begin{equation*}
2Sf+(\dfrac{n-1}{n-2}+\dfrac19)f^2_3> (1-\dfrac2n)S^3.
\end{equation*}
then 
 \begin{equation}\label{eq:2.4}
|{\vec \varphi_{*}}|^2\le\frac59\dfrac{f}{S^2}.
\end{equation}
\end{lemma}
\begin{proof}
Let $\Lambda=\diag(\dfrac{\lambda_1}{\sqrt S},\ldots,\dfrac{\lambda_n}{\sqrt S})$ and put
\[
 e_1=\frac1{\sqrt n}(1,\ldots,1)^T,\qquad
 e_2=\left(\frac{\lambda_1}{\sqrt S},\ldots,\frac{\lambda_n}{\sqrt S}\right)^T.
\]
Assume first that $f>0$. We define
\[
 e_3=\left(\frac{a_{11}}{\sqrt f},\ldots,\frac{a_{nn}}{\sqrt f}\right)^T.
\]
$e_1, \ e_2, \ e_3$ are orthonormal.
Since
\begin{equation}\label{eq:2.5}
 \Lambda e_1=\frac{e_2}{\sqrt n},\qquad
 \Lambda e_2=\frac{e_1}{\sqrt n}+\dfrac{f_3}{S^{3/2}}e_2+\dfrac{\sqrt f}S\,e_3, \qquad
 \Lambda e_3=\dfrac{\sqrt f}S\,e_2+ae_3+\vec b,
\end{equation}
where $\vec b\in \mathcal H=\{e_1, e_2, e_3\}^{\perp}$, 
the matrix of $\Lambda$ relative to $e_1, e_2, e_3$  and their
orthogonal complement $\mathcal H$ has the form
\[
 \Lambda=
 \begin{pmatrix}
 0&\frac{1}{\sqrt n}&0&0\\
 \frac{1}{\sqrt n}&\dfrac{f_3}{S^{3/2}}&\dfrac{\sqrt f}S&0\\
 0&\dfrac{\sqrt f}S&a&\vec b^{T}\\
 0&0&\vec b&C
 \end{pmatrix}.
\]
Here $C$ is symmetric, and $\dim\mathcal H=n-3$.
From $\tr\Lambda=0$ and $|\Lambda|^2=1$, we have
\begin{equation}\label{eq:2.6}
 \tr C=-\dfrac{f_3}{S^{3/2}}-a,\qquad
 1=\frac{2}{n}+\dfrac{f_3^2}{S^3}+2\dfrac{ f}{S^2}+a^2+2|\vec b|^2+|C|^2.
\end{equation}
Setting  $C_0=C+\frac{\frac{f_3}{S^{3/2}}+a}{n-3}I_{\mathcal H}$, where   $I_{\mathcal H}$ denotes  the unit matrix in ${\mathcal H}$, we have 
 $$
 \tr C_0=0, \quad |C|^2=|C_0|^2+\frac{(\frac{f_3}{S^{3/2}}+a)^2}{n-3}.
 $$
We obtain from (\ref{eq:2.6}) and the above formula
\begin{equation}\label{eq:2.7}
\begin{aligned}
 &S^3-2Sf-\dfrac{n-1}{n-2}f^2_3\\
 &=\dfrac{2}nS^3+S^3a^2+\frac{(f_3+S^{3/2}a)^2}{n-3}-\frac{f_3^2}{n-2}
              +2S^3|b|^2+S^3|C_0|^2\\
 &=\dfrac{2}nS^3+\frac{n-2}{n-3}\left(S^{3/2}a+\frac{f_3}{n-2}\right)^2
              +2S^3|b|^2+S^3|C_0|^2\ge\dfrac{2}nS^3,
\end{aligned}
\end{equation}
that is,
\begin{equation*}
2Sf+\dfrac{n-1}{n-2}f^2_3\leq (1-\dfrac2n)S^3.
\end{equation*}
For $f=0$, 
\[
 \Lambda e_1=\frac{e_2}{\sqrt n},\qquad
 \Lambda e_2=\frac{e_1}{\sqrt n}+\dfrac{f_3}{S^{3/2}}e_2.
 \]
Hence, the span of $e_1,e_2$ is invariant under $\Lambda$.
Its orthogonal complement has dimension $n-2$ and a symmetric
block $C$ with $\tr C=-\dfrac{f_3}{S^{3/2}}$.
Thus
$$
1=\frac{2}{n}+\frac{f_3^2}{S^3}+|C|^2
\ge\frac{2}{n}+\frac{n-1}{n-2}\frac{f_3^2}{S^3},
$$
which  proves (\ref{eq:2.3}).\\
If $f=0$, then $\vec\varphi=\vec\varphi_*=0$, and the second assertion follows.
We only consider $f>0$. If 
\begin{equation*}
2Sf+(\dfrac{n-1}{n-2}+\dfrac19)f^2_3>(1-\dfrac2n)S^3,
\end{equation*}
according to (\ref{eq:2.7}), we infer
\begin{equation}\label{eq:2.8}
\begin{aligned}
 &\frac19f_3^2>\frac{n-2}{n-3}\left(S^{3/2}a+\frac{f_3}{n-2}\right)^2
              +2S^3|b|^2+S^3|C_0|^2.
\end{aligned}
\end{equation}
We obtain $f_3\neq 0$. 
Define
$$
\begin{aligned}
&\vec q_1=\frac{1}{S}(\lambda_1^2, \cdots, \lambda_n^2)^T,  \quad  \vec q_2=(\dfrac{\lambda_1^3}{S\sqrt S},\ldots,\dfrac{\lambda_n^3}{S\sqrt S})^T.
\end{aligned}
$$
We know that 
\begin{equation}
 \begin{aligned}
 &\vec q_1=\Lambda e_2=\frac{e_1}{\sqrt n}+\dfrac{f_3}{S^{3/2}}e_2+\dfrac{\sqrt f}S\,e_3, \\
&\vec q_2= \Lambda q_1=\dfrac{f_3}{S\sqrt{nS}}e_1+(\dfrac1n+\dfrac{f_3^2}{S^{3}}+\dfrac{f}{S^2})e_2
+\dfrac{\sqrt f}S (\dfrac{f_3}{S^{3/2}}+a) e_3+\dfrac{\sqrt f}S\vec b,
\end{aligned}
\end{equation}
$$
{\vec q}=S\sqrt S\vec q_2- \dfrac75\sqrt Sf_3\vec q_1.
$$
The projection ${\vec \psi }$ of $\vec q$  onto $\{e_1, e_2\}^{\perp}$  is given by
$$
{\vec \psi}=\sqrt {Sf} (a-\dfrac25\dfrac{f_3}{S^{3/2}}) e_3+\sqrt {Sf}\vec b.
$$
We obtain
$$
|{\vec \psi}|^2=Sf\bigl[(a-\dfrac25\dfrac{f_3}{S^{3/2}})^2+|\vec b|^2\bigr], \quad 
\langle \dfrac{\sqrt f}Se_3, {\vec \psi}\rangle=\dfrac{f}{\sqrt S}(a-\dfrac25\dfrac{f_3}{S^{3/2}})
$$
Since ${\vec\varphi}=\dfrac{\sqrt f}Se_3$, if $\vec\psi\ne0$, we have
\begin{equation}\label{eq:2.10}
 |\vec\varphi_*|^2
 =|\vec\varphi|^2-\frac{\langle\vec\varphi,\vec\psi\rangle^2}{|\vec\psi|^2}
 =\frac{f}{S^2}
 \frac{|b|^2}{(a-\frac25\frac{f_3}{S^{3/2}})^2+|b|^2}.
\end{equation}
If $\vec\psi=0$, then
$a=\frac25\frac{f_3}{S^{3/2}}$ and $b=0$.
Assume that $|\vec\varphi_*|^2\ge\frac59\frac{f}{S^2}$.
By \eqref{eq:2.10}, we obtain, when $\vec\psi\ne0$,
\[
 |b|^2\ge\frac54\left(a-\frac25\frac{f_3}{S^{3/2}}\right)^2.
\]
This inequality also holds when $\vec\psi=0$.
Hence, from \eqref{eq:2.8} and $(n-2)/(n-3)\ge1$, we infer
\begin{equation*}
 \frac19 f_3^2>
 \left(S^{3/2}a+\frac{f_3}{n-2}\right)^2
 +\frac52\left(S^{3/2}a-\frac25f_3\right)^2.
\end{equation*}
On the other hand, a direct computation yields
\begin{align*}
 &\left(S^{3/2}a+\frac{f_3}{n-2}\right)^2
 +\frac52\left(S^{3/2}a-\frac25f_3\right)^2\\
 &=\frac72\left[S^{3/2}a+
       \frac27\left(\frac1{n-2}-1\right)f_3\right]^2
   +\frac57\left(\frac1{n-2}+\frac25\right)^2f_3^2\\
 &\ge\frac4{35}f_3^2>\frac19f_3^2,
\end{align*}
since $f_3\ne0$ and $4/35-1/9=1/315>0$.
This is a contradiction. Thus $|\vec\varphi_*|^2<\frac59\frac{f}{S^2}$ for $f>0$,
which proves \eqref{eq:2.4}.
\end{proof}

The following  formulas can be found in  Peng-Terng \cite{pt1} and Yang-Cheng \cite{yc3}:
\begin{equation}\label{eq:2.12}
\begin{aligned}
\begin{cases}
& \Delta h_{ij}=(n-S)h_{ij},\qquad \sum_{i,j,k}h_{ijk}^2=S(S-n)=tS^2,\\
 &\sum_{i,j,k,l}h_{ijkl}^2
 =\bigl[(2+\frac{3}{n})t-(1+\frac{3}{n})\bigr]tS^3+3(A-2B),
\end{cases}
\end{aligned}
\end{equation}
where 
\[
  A=\sum_{ijk}\lambda_i^2h_{ijk}^2,\qquad
 B=\sum_{ijk}\lambda_i\lambda_jh_{ijk}^2,\qquad
  C=\sum_{ijk}\lambda_ih_{ijk}^2.
\]
Since  $X:M^n\to\mathbb S^{n+1}(1)$  is minimal and $S$ is constant, we have
\begin{equation}\label{eq:2.13}
\begin{aligned}
& \sum_ih_{iik}=0,\qquad \sum_i\lambda_ih_{iik}=0,\\
& \sum_{ij}h_{ij}h_{ijkl}=-\sum_{ij}h_{ijk}h_{ijl},\qquad
 \sum_{ijk}h_{ijk}h_{ijkl}=0.
\end{aligned}
\end{equation}
Put $D_k=\sum_i\lambda_i^2h_{iik}$.
By a direct computation, we know 
\begin{equation}\label{eq:2.14}
\begin{aligned}
 &\dfrac13\Delta f_3=(n-S)f_3+2C,\\
& \dfrac14\Delta f_4=(n-S)f_4+(2A+B)\\
\end{aligned}
\end{equation}
The following lemma can be proved by making use of the  idea  in Yang-Cheng \cite{yc3}. For the reader's convenience, we give a proof.
\begin{lemma}\label{lem:2.2} 
For a minimal hypersurface with constant scalar curvature, 
one has, at each point,
\begin{equation}\label{eq:cubic}
 A+2B\ge \frac{3C^2}{tS^2}
 +4\sum_l\frac{D_l^2}{S+\frac{2n}{n+2}\lambda_l^2}.
\end{equation}
\end{lemma}
\begin{proof}
Since, for any $y_k$,
$$
\sum_{i,j,k}h_{ij}h_{ijk}y_k=0,
$$
we have, taking  $a_k=\dfrac{2\lambda_k}{n+2}$, $y_k=\dfrac{2}{S+\dfrac{2n}{n+2}\lambda_k^2}\sum_{i}\lambda_i^2h_{iik}$
$$
\begin{aligned}
&\bigl(3\sum_{i,j,k}\lambda_ih_{ijk}^2\bigr)^2\\
&=\bigl(\sum_{i,j,k}\bigl\{(\lambda_i+\lambda_j+\lambda_k)h_{ijk}-h_{ij}y_k-h_{jk}y_i-h_{ik}y_j+\delta_{ij}a_ky_k
+\delta_{ik}a_jy_j+\delta_{jk}a_iy_i\bigr\}h_{ijk}\bigr)^2\\
&\leq \sum_{i,j,k}\bigl\{(\lambda_i+\lambda_j+\lambda_k)h_{ijk}-h_{ij}y_k-h_{jk}y_i-h_{ik}y_j
+\delta_{ij}a_ky_k+\delta_{ik}a_jy_j+\delta_{jk}a_iy_i\bigr\}^2\sum_{i,j,k}h_{ijk}^2\\
&=\bigl\{3(A+2B)-12\sum_k\sum_{i}\lambda_i^2h_{iik}y_k+3\sum_ky_k^2(S+2\lambda_k^2)-\sum_k\dfrac{12}{n+2}\lambda_k^2y_k^2\bigr\}S(S-n)\\
&=\bigl\{3(A+2B)-12\sum_{k}\dfrac{(\sum_i\lambda_i^2h_{iik})^2}{S+\dfrac{2n}{n+2}\lambda_k^2}\bigr\}S(S-n).\\
\end{aligned}
$$
We obtain
$$
\begin{aligned}
&3\bigl(\sum_{i,j,k}\lambda_ih_{ijk}^2\bigr)^2
\leq \bigl\{(A+2B)-4\sum_{k}\dfrac{n+2}{(n+2)S+2n\lambda_k^2}(\sum_i\lambda_i^2h_{iik})^2\bigr\}tS^2.\\
\end{aligned}
$$
\end{proof}

We retain the notation of the lemma 2.1. In particular, we put
\[
 e_1=\frac1{\sqrt n}(1,\ldots,1)^T,\qquad
 e_2=\left(\frac{\lambda_1}{\sqrt S},\ldots,\frac{\lambda_n}{\sqrt S}\right)^T.
\]
For any vector $\vec w=(w_1,\ldots,w_n)^T$, define
\[
 B_{ij}(w)=\sum_kh_{ijk}w_k,\quad B(w)=(B_{ij}(w)),\quad
 G_{kl}=\sum_{i,j}h_{ijk}h_{ijl},\quad D=\sum_k\lambda_k^2D_k^2.
\]

\begin{lemma}\label{lem:2.3}
If $\sum_i(S\lambda_i^3-\dfrac{7}5f_3\lambda_i^2)h_{iik}=0$, for any $k$, then,
\begin{gather}
 (n+2)S\sum_{i,j}(\sum_kh_{ijk}w_k)^2+2n\sum_{i,j}(\sum_kh_{ijk}\lambda_kw_k)^2\le ntS^3\sum_kw_k^2,
 \label{eq:gram}\\
 \sum_i\left(\frac{n+2}{n}S+2\lambda_i^2\right)D_i^2\le tS^5|{\vec \varphi}_*|^2,
 \label{eq:budget}
\end{gather}
for arbitrary $w_k$.
\end{lemma}

\begin{proof}
The symmetric matrix
\[
 G_{kl}=\sum_{i,j}h_{ijk}h_{ijl}
\]
is positive semidefinite, because
\[
 \sum_{k,l}G_{kl}w_kw_l
 =\sum_{i,j}\sum_{k,l}h_{ijk}h_{ijl}w_kw_l
 =\sum_{i,j}\left(\sum_kh_{ijk}w_k\right)^2\ge0.
\]
Choose an orthonormal eigenbasis $\vec u_1, \ldots, \vec u_n$ for $G$, with eigenvalues
$g_k\ge0$. By a direct computation, we have
\begin{align*}
 \langle B(u_a),B(u_b)\rangle
 &=\sum_{i,j}\left(\sum_kh_{ijk}(u_a)_k\right)
                   \left(\sum_lh_{ijl}(u_b)_l\right)\\
 &=\sum_{k,l}G_{kl}(u_a)_k(u_b)_l
 =\langle \vec u_a,G(\vec u_b)\rangle=g_b\delta_{ab},
\end{align*}
where $(u_a)_k$ denotes the $k$th component of $\vec u_a$.
In particular, $|B(u_a)|^2=g_a$, and
\[
 \sum_a g_a=\tr G=\sum_k G_{kk}=\sum_{i,j,k}h_{ijk}^2=tS^2.
\]
From $\sum_i h_{iik}=\sum_i\lambda_i h_{iik}=0$, we have
\begin{align*}
 \langle I_E,B(w)\rangle
 &=\sum_k w_k\sum_i h_{iik}=0,\\
 \langle h,B(w)\rangle
 &=\sum_k w_k\sum_i\lambda_i h_{iik}=0.
\end{align*}
Here $h=(h_{ij})$.
The symmetric matrices
\[
 \frac {I_E}{\sqrt n},\qquad \frac {h}{\sqrt S},\qquad
 \frac{B(u_b)}{\sqrt{g_b}}\quad(g_b>0)
\]
are orthonormal in the Frobenius inner product, where $I_E$ denotes the unit matrix.\\
For each $a$, put
\begin{align*}
 R_a={}&\frac12(\vec w\otimes \vec u_a+\vec u_a\otimes \vec w)
 -\frac{\langle \vec w, \vec u_a\rangle}{n}I_E
 -\frac{\langle h(\vec w), \vec u_a\rangle}{S}h\\
 &-\sum_{b:g_b>0}\frac{\langle B(u_b)\vec w, \vec u_a\rangle}{g_b}B(u_b).
\end{align*}
The matrices subtracted here are mutually orthogonal. For any symmetric
matrix $K$,
\[
 \left\langle\frac12(\vec w\otimes \vec u_a+\vec u_a\otimes \vec w),K\right\rangle
 =\langle K(\vec w), \vec u_a\rangle,
\]
and
\[
 \left|\frac12(\vec w\otimes \vec u_a+\vec u_a\otimes \vec w)\right|^2
 =\frac12|\vec w|^2+\frac12\langle \vec w, \vec u_a\rangle^2.
\]
By making use of these orthogonality relations, a direct computation yields
\begin{align}
 |R_a|^2={}&\frac12|\vec w|^2+
 \left(\frac12-\frac1n\right)\langle \vec w, \vec u_a\rangle^2
 -\frac{\langle h(\vec w), \vec u_a\rangle^2}{S}\notag\\
 &-\sum_{b:g_b>0}\frac{\langle B(u_b) \vec w, \vec u_a\rangle^2}{g_b}.
 \label{gradient:remainder}
\end{align}
For any vector $\vec v$, we also have
\[
 |B(v)|^2=\sum_a g_a\langle \vec v, \vec u_a\rangle^2.
\]
Full symmetry of $h_{ijk}$ gives
\[
 \langle B(u_b)(\vec w), \vec u_a\rangle
 =\sum_{i,j,k}h_{ijk}(u_a)_iw_j(u_b)_k
 =\langle B(w)(\vec u_b), \vec u_a\rangle.
\]
The latter expression is symmetric in $a,b$. If $g_a=0$, then
$B(u_a)=0$, and symmetry shows that these entries vanish whenever either
index has zero eigenvalue. Hence, by interchanging $a$ and $b$, we obtain
\begin{align*}
 &\sum_{a,b:g_a,g_b>0}\frac{g_a}{g_b}
       \langle B(w)(\vec u_b), \vec u_a\rangle^2\\
 &=|B(w)|^2+
 \sum_{a<b:g_a,g_b>0}\frac{(g_a-g_b)^2}{g_ag_b}
       \langle B(w)(\vec u_b), \vec u_a\rangle^2,
\end{align*}
since $g_a/g_b+g_b/g_a=2+(g_a-g_b)^2/(g_ag_b)$.
Multiplying \eqref{gradient:remainder} by $g_a$ and summing over $a$,
we obtain, since $\sum_a g_a=tS^2$,
\begin{align*}
 &ntS^3|w|^2-(n+2)S|B(w)|^2-2n|B(h(\vec w))|^2\\
 &\quad=2nS\sum_a g_a|R_a|^2
 +2nS\sum_{a<b:g_a,g_b>0} \frac{(g_a-g_b)^2}{g_ag_b}\langle B(w)(\vec u_b), \vec u_a\rangle^2\ge0.
\end{align*}
Since $(h(\vec w))_k=\lambda_kw_k$ in the principal frame, this is
\eqref{eq:gram}.\\
Since, for any $k$,
\[
 S\sum_i b_{ii}h_{iik}
 =\sum_i\left(S\lambda_i^3-\frac75f_3\lambda_i^2\right)h_{iik}=0,
\]
the vector $(h_{11k},\ldots,h_{nnk})^T$ is orthogonal to
$\one$, $(\lambda_1,\ldots,\lambda_n)^T$ and $\vec q$.
According to \eqref{eq:projection-definition}, we obtain
\[
 D_k=\sum_i a_{ii}h_{iik}=S\sum_i(\vec\varphi_*)_i h_{iik}.
\]
Thus, from Cauchy inequality,
\[
 D_k^2\le S^2|\vec\varphi_*|^2\sum_i h_{iik}^2.
\]
Hence, we get 
\begin{equation}\label{r:D}
 D_k^2\le S^2|{\bf \varphi_*}|^2\sum_ih_{iik}^2.
\end{equation}
Let $E_i$ now denote the $i$th coordinate vector in the principal frame.
The projection of $E_i$ onto the orthogonal complement of
$e_1$ and $e_2$ is
\[
\vec \eta= E_i-\frac1{\sqrt n}e_1
 -\frac{\lambda_i}{\sqrt S}e_2.
\]
We infer 
\[
 |\vec\eta|^2=|E_i|^2-\left(\frac1{\sqrt n}\right)^2
 -\left(\frac{\lambda_i}{\sqrt S}\right)^2
 =1-\frac1n-\frac{\lambda_i^2}{S}.
\]
Since $\sum_jh_{jjk}=\sum_j\lambda_jh_{jjk}=0$, we have
\[
 h_{kkk}=\sum_j\left(\delta_{kj}-\frac1n
                   -\frac{\lambda_k\lambda_j}{S}\right)h_{jjk}.
\]
It follows from Cauchy's inequality that
\[
 h_{kkk}^2\le\left(1-\frac1n-\frac{\lambda_k^2}{S}\right)
 \sum_jh_{jjk}^2.
\]
Thus, we can
\begin{align*}
 tS^2&=\sum_i h_{iii}^2+3\sum_{i\ne j}h_{jji}^2
       +6\sum_{i<j<k}h_{ijk}^2\\
 &\ge\sum_i\left(3\sum_jh_{jji}^2-2h_{iii}^2\right)\\
 &\ge\sum_i\left[3-2\left(1-\frac1n-\frac{\lambda_i^2}{S}\right)\right]
       \sum_jh_{jji}^2\\
 &=\sum_i\left(\frac{n+2}{n}+\frac{2\lambda_i^2}{S}\right)
       \sum_jh_{jji}^2.
\end{align*}
Multiplying the last inequality by $S$ gives
\[
 \sum_i\left(\frac{n+2}{n}S+2\lambda_i^2\right)\sum_jh_{jji}^2
 \le tS^3.
\]
Combining this inequality with \eqref{r:D}, we obtain
\begin{align*}
 \sum_i\left(\frac{n+2}{n}S+2\lambda_i^2\right)D_i^2
 &\le S^2|\vec \varphi_*|^2\sum_i\left(\frac{n+2}{n}S+2\lambda_i^2\right)
       \sum_jh_{jji}^2\\
 &\le tS^5|\vec \varphi_*|^2.
\end{align*}
This proves \eqref{eq:budget}. If $|\vec \varphi_*|^2=0$, \eqref{r:D} gives $D_i=0$
for every $i$, so \eqref{eq:budget} also holds in this case.
\end{proof}

\section{Main estimates}

For a minimal hypersurface $X:M^n\to\mathbb S^{n+1}(1)$, 
we use the definition of $u_{ijkl}$ introduced by Yang-Cheng \cite{yc3}.
\[
 u_{ijkl}=\frac{h_{ijkl}+h_{ijlk}+h_{iklj}+h_{jkli}}{4}.
\]
 In view of the Ricci identity,  a direct computation yields
 \begin{equation*}
\begin{aligned}
  \sum_{i,j,k,l}h_{ijkl}^2 = \sum_{i,j,k,l}u_{ijkl}^2 +\frac32\bigl( Sf_4-f_3^2-\dfrac{S^3}n+\dfrac{S(S-n)^2}n\bigr),
\end{aligned}
\end{equation*}
which was proved in \cite{cwy}. According to (\ref{eq:2.12}), we have
\begin{equation}\label{eq:3.1}
\begin{aligned}
  \sum_{i,j,k,l}u_{ijkl}^2 = 3(A-2B)-\frac32Sf +\bigl[(2+\frac{3}{2n})t-(1+\frac{3}{n})\bigr]tS^3.
\end{aligned}
\end{equation}
Since $S$ is constant, 
the Ricci identities, (\ref{eq:2.12}) and (\ref{eq:2.13}) give
\begin{equation}
\begin{aligned}
& \sum_iu_{iikl}=-\frac{t}{2}S\lambda_k\delta_{kl},\\
& \sum_i\lambda_iu_{iikl}
 =-\sum_{i,j}h_{ijk}h_{ijl}
       +\frac{1}{2}\left(a_{kk}+\frac{tS}{n}\right)S\delta_{kl},
       \label{eq:U-traces}\\
 &\sum_{i,j,k}h_{ijk}u_{ijkl}=-\frac{3}{2}\lambda_l D_l.
\end{aligned}
\end{equation}
Define, for $\vec b=(b_1, \cdots, b_n)^T$,
\begin{equation}
\begin{aligned}
 &V_{ijkl}(b)=b_ih_{jkl}+b_jh_{kli}+b_kh_{lij}+b_lh_{ijk}, \quad X_{ijkl}=h_{ij}a_{kl}+a_{ij}h_{kl},\\
 &Y_{ijkl}=\dfrac{S}3(h_{ij}\delta_{kl}+h_{kl}\delta_{ij}+h_{ik}\delta_{jl}+h_{il}\delta_{jk}+h_{kj}\delta_{il}+h_{jl}\delta_{ik}), \quad Z_{ijkl}=h_{ij}h_{kl}.
\end{aligned}
 \end{equation}
Since $S$ is constant, we obtain the following by a direct computation:
\begin{lemma}\label{lem:3.1}
$$
\begin{aligned}
& \sum_{i,j,k,l}X_{ijkl}^2=2S(f_4-\dfrac{f_3^2}S-\dfrac{S^2}n)=2Sf, \quad 
 \sum_{i,j,k,l}Y_{ijkl}^2=\frac{2(n+4)}3S^3,  \quad
 \sum_{i,j,k,l}Z_{ijkl}^2=S^2,\\
& \sum_{i,j,k,l}X_{ijkl}Y_{ijkl}=\frac83Sf,\quad \sum_{i,j,k,l}X_{ijkl}Z_{ijkl}=0, \quad
\sum_{i,j,k,l}X_{ijkl}V_{ijkl}(b)=4\sum_lb_l\lambda_lD_l,\\
& \sum_{i,j,k,l}X_{ijkl}u_{ijkl}=-2A+Sf+\dfrac{2f_3C}S+\dfrac{2tS^3}n, \quad \sum_{i,j,k,l}Z_{ijkl}Y_{ijkl}=\frac43Sf_3, \quad 
\sum_{i,j,k,l}V_{ijkl}(b)Y_{ijkl}=0,\\ 
 & \sum_{i,j,k,l}u_{ijkl}Y_{ijkl}=-tS^3, \quad \sum_{i,j,k,l}V_{ijkl}(b)Z_{ijkl}=0, \quad  \sum_{i,j,k,l}u_{ijkl}Z_{ijkl}=-C,\\
 &\sum_{i,j,k,l}u_{ijkl}V_{ijkl}(b)=-6\sum_lb_l\lambda_lD_l,\quad 
 \sum_{i,j,k,l}V_{ijkl}^2(b)=4\sum_{i,j,k}h_{ijk}^2\sum_ib_i^2+12\sum_{i,j}G_{ij}b_ib_j.
\end{aligned}
$$
\end{lemma}
\begin{theorem}
For a minimal hypersurface $X:M^n\to\mathbb S^{n+1}(1)$ with constant scalar curvature and $S>n$, we have, for any $u>0$,
\begin{equation}\label{eq:3.5}
\begin{aligned}
 &3(A-2B)+4uA-(\frac32+\dfrac{4t}{n}u-2u^2+2u)Sf -4uf_3\dfrac{C}S\\
 &+\bigl[(2+\frac{3}{8n}+\dfrac{3}{2n^2})t-(1+\frac{3}{n}+\dfrac{4u}{n})\bigr]tS^3
 -(\dfrac{56}{9t}+\dfrac{16}{3t}u)\dfrac{D}{S^2}\\
 &+\dfrac{16}{3t^2S^4}\sum_{i,j}(\sum_{k}h_{ijk}\lambda_kD_k)^2
 -\dfrac{1}{S^2}(C-\frac{tS}{n}f_3)^2\geq 0.
\end{aligned}
\end{equation}
\end{theorem}
\begin{proof}
Since
\[
\sum_{i,j,k,l}(u_{ijkl}-uX_{ijkl}+\dfrac{2}{3t}V_{ijkl}(b)+\dfrac{3t}{4n} Y_{ijkl})Z_{ijkl}
= -C+\frac{tS}{n}f_3, 
\]
we obtain from Cauchy inequality
\begin{equation}
\sum_{i,j,k,l}(u_{ijkl}-uX_{ijkl}+\dfrac{2}{3t}V_{ijkl}(b)+\dfrac{3t}{4n} Y_{ijkl})^2 \geq \dfrac{1}{S^2}(C-\frac{tS}{n}f_3)^2.
\end{equation}
Taking
$b_l=\frac{\lambda_lD_l}{S^2}$, from the lemma \ref{lem:3.1}, 
a long and direct computation yields
\begin{equation}\label{eq:3.6}
\begin{aligned}
\sum_{i,j,k,l}u_{ijkl}^2&\geq (\dfrac{4t}{n}u-2u^2+2u)Sf-4uA+4uf_3\dfrac{C}S\\
 &+(\dfrac{56}{9t}+\dfrac{16}{3t}u)\dfrac{D}{S^2}-\dfrac{16}{3t^2S^4}\sum_{i,j}(\sum_{k}h_{ijk}\lambda_kD_k)^2\\
 &+\bigl(\dfrac{3(3n-4)}{8n^2}t^2+\dfrac{4u}{n}t\bigr)S^3
 +\dfrac{1}{S^2}(C-\frac{tS}{n}f_3)^2.
\end{aligned}
\end{equation}
(\ref{eq:3.1}) and (\ref{eq:3.6}) give
\begin{equation*}
\begin{aligned}
 &3(A-2B)+4uA-(\frac32+\dfrac{4t}{n}u-2u^2+2u)Sf -4uf_3\dfrac{C}S\\
 &+\bigl[(2+\frac{3}{8n}+\dfrac{3}{2n^2})t-(1+\frac{3}{n}+\dfrac{4u}{n})\bigr]tS^3
 -(\dfrac{56}{9t}+\dfrac{16}{3t}u)\dfrac{D}{S^2}\\
 &+\dfrac{16}{3t^2S^4}\sum_{i,j}(\sum_{k}h_{ijk}\lambda_kD_k)^2
 -\dfrac{1}{S^2}(C-\frac{tS}{n}f_3)^2\geq 0.
\end{aligned}
\end{equation*}
\end{proof}
Following Yang-Cheng \cite{yc3}, take
\begin{equation}\label{eq:fixed-function}
 F=\frac{S}{4}f_4-\frac{7}{30}f_3^2.
\end{equation}
Since $X:M^n\to\mathbb S^{n+1}(1)$ is a complete  minimal hypersurface, 
and  $S$ is constant, we know that the Ricci curvature is bounded from below according to  the Gauss equations.
Since $f_4\le S^2$ and $f_3^2\le S^3$, the function $F$ is bounded.
By applying  the Generalized Maximum Principle of  Omori \cite{o} and Yau \cite{y} to the function $F$ in \eqref{eq:fixed-function},  
we know that 
there exists a sequence of $\{p_m\}_{m=1}^{\infty} \subset  M^n$  such that 
\begin{equation}\label{eq:3.8}
\lim_{m\to \infty}F(p_m)=\sup F,  \quad \lim_{m\to \infty}|\nabla F(p_m)|=0,  \quad \limsup_{m\to \infty}\Delta F(p_m)\leq 0.
\end{equation}
Since $S$ is constant, we know that, for any $i, j, k, l$,  in view of (\ref{eq:2.12}),
 $\{\lambda_i(p_m)\}$, $\{h_{ijk}(p_m)\}$ and $\{h_{ijkl}(p_m)\}$ are bounded sequences.
Thus, we can assume,  for any $i, j, k, l$, 
 $$
 \lim_{m\to\infty}\lambda_i(p_m)=\bar{\lambda}_i, \ \  \lim_{m\to\infty} h_{ijk}(p_m) =\bar{h}_{ijk},\quad
 \lim_{m\to\infty} h_{ijkl}(p_m)=\bar{h}_{ijkl}
 $$
 if necessary we can take a subsequence.
All of the following computations are made for $\bar{\lambda}_i$, $\bar{h}_{ijk}$ and $\bar{h}_{ijkl}$. For simplicity, we
 omit $\bar{ }$.  
 We should emphasize that from now on,  our symbols denote those ones under limit.
The vectors $\vec\psi$ and $\vec\varphi_*$ are defined by
\eqref{eq:projection-definition} using the limiting principal curvatures.\\
 Since $tS=S-n$,  if $t>\frac13$, we know $S>n+\frac12n$. We can assume $0<t\leq\frac13$.\\
Define $\mathcal R_u$, $\mathcal E_u$ and $\mathcal G_u$ by, respectively,
\begin{equation}
\begin{aligned}
  \mathcal R_u=&-(\dfrac{56}{9t}+\dfrac{16}{3t}u)\dfrac{D}{S^2}+\frac{21}{5}\left(4+\frac{8u}{3}\right)\dfrac{\sum_lD_l^2}S\\
 &+\dfrac{16}{3t^2S^4}\sum_{i,j}(\sum_{k}h_{ijk}\lambda_kD_k)^2
 -4(5+\frac{4u}{3})\sum_l\frac{D_l^2}{S+\frac{2n}{n+2}\lambda_l^2},
\end{aligned}
\end{equation}
\begin{equation}
\begin{aligned}
\mathcal E_u={}
 &\bigl[-\frac32+(4+\dfrac{8u}3-\dfrac{4u}{n})t+2u^2-2u\bigr]Sf -\frac{2}{5}\left(4+\frac{8u}{3}\right)tf_3^2\\
&+\left(\frac{14}{5}\bigl(4+\frac{8u}{3}\bigr)-4u\right)\dfrac{f_3}SC  -\dfrac{15+4u}t \frac{C^2}{S^2}  
 -\dfrac{1}{S^2}(C-\frac{tS}{n}f_3)^2,\\
\end{aligned}
\end{equation}
$$
\mathcal G_{u}=\bigl[(2+\frac{3}{8n}+\dfrac{3}{2n^2})t-(1-\frac{1}{n}+\dfrac{4u}{3n})\bigr]tS^3.
$$
\begin{lemma}
For the above limiting quantities and $u>0$, we have
\begin{equation}
0\leq \mathcal R_u+\mathcal E_u+\mathcal G_u.
\end{equation}
 \end{lemma}
 \begin{proof}
 According to (\ref{eq:2.14}) and  (\ref{eq:3.8}), we infer
\begin{equation}\label{eq:3.12}
\begin{aligned}
&\lim_{m\to \infty}|\nabla F(p_m)|^2=\sum_k\left[\sum_i\left(S\lambda_i^3-\dfrac75f_3\lambda_i^2\right)h_{iik}\right]^2=0,\\
&\limsup_{m\to \infty}\Delta F(p_m)= -tS^2f_4+S(2A+B)-\dfrac75\bigl(-tSf_3^2+2f_3C +3\sum_lD_l^2\bigr)\leq 0.
\end{aligned}
\end{equation}
$(\ref{eq:3.5}) -(4+\frac{8u}{3})\frac{1}{S}\times$ the second formula in  (\ref{eq:3.12}) yields, for $u>0$, 
\begin{equation}\label{eq:3.13}
\begin{aligned}
 &-(5+\frac{4u}{3})(A+2B)+\bigl[-\frac32+(4+\dfrac{8u}3-\dfrac{4u}{n})t+2u^2-2u\bigr]Sf \\
 &+\bigl[(2+\frac{3}{8n}+\dfrac{3}{2n^2})t-(1-\frac{1}{n}+\dfrac{4u}{3n})\bigr]tS^3
 -\frac{2}{5}\left(4+\frac{8u}{3}\right)tf_3^2\\
 &+\left(\frac{14}{5}\bigl(4+\frac{8u}{3}\bigr)-4u\right)\dfrac{f_3}SC 
 -(\dfrac{56}{9t}+\dfrac{16}{3t}u)\dfrac{D}{S^2}+\frac{21}{5}\left(4+\frac{8u}{3}\right)\dfrac{\sum_lD_l^2}S\\
 &+\dfrac{16}{3t^2S^4}\sum_{i,j}(\sum_{k}h_{ijk}\lambda_kD_k)^2
 -\dfrac{1}{S^2}(C-\frac{tS}{n}f_3)^2\geq 0.
\end{aligned}
\end{equation}
Hence, we conclude according to the lemma \ref{lem:2.2} and (\ref{eq:3.13})
\begin{equation*}
\begin{aligned}
 &0\leq \bigl[-\frac32+(4+\dfrac{8u}3-\dfrac{4u}{n})t+2u^2-2u\bigr]Sf-\dfrac{15+4u}t \frac{C^2}{S^2} \\
 &+\bigl[(2+\frac{3}{8n}+\dfrac{3}{2n^2})t-(1-\frac{1}{n}+\dfrac{4u}{3n})\bigr]tS^3
 -\frac{2}{5}\left(4+\frac{8u}{3}\right)tf_3^2\\
 &+\left(\frac{14}{5}\bigl(4+\frac{8u}{3}\bigr)-4u\right)\dfrac{f_3}SC 
 -(\dfrac{56}{9t}+\dfrac{16}{3t}u)\dfrac{D}{S^2}+\frac{21}{5}\left(4+\frac{8u}{3}\right)\dfrac{\sum_lD_l^2}S\\
 &+\dfrac{16}{3t^2S^4}\sum_{i,j}(\sum_{k}h_{ijk}\lambda_kD_k)^2
 -\dfrac{1}{S^2}(C-\frac{tS}{n}f_3)^2-4(5+\frac{4u}{3})\sum_l\frac{D_l^2}{S+\frac{2n}{n+2}\lambda_l^2}.
\end{aligned}
\end{equation*}
The definitions of $\mathcal R_u$, $\mathcal E_u$ and $\mathcal G_u$ give
$$
0\leq \mathcal R_u+\mathcal E_u+\mathcal G_u.
$$
 \end{proof}
 \begin{lemma}\label{lem:3.4}
For the above limiting quantities, if $|\vec\varphi_*|>0$, then, for $u>0$,
\begin{equation}
\begin{aligned}
 &\mathcal R_u\leq -(\dfrac{8}{9}+\dfrac{16}{3}u+\dfrac{32}{3(n+2)})\dfrac{D}{tS^2}+\frac{21}{5}\left(4+\frac{8u}{3}\right)\dfrac{\sum_lD_l^2}S\\
 &-\dfrac{32n}{3(n+2)}\dfrac{D^2}{t^2S^7|{\vec \varphi_{*}}|^2}
 -4(5+\frac{4u}{3})\sum_l\frac{D_l^2}{S+\frac{2n}{n+2}\lambda_l^2}.\\
\end{aligned}
\end{equation}
 \end{lemma}
 \begin{proof}
In view of the lemma \ref{lem:2.3}, taking $w_k=\lambda_kD_k$, we obtain
\begin{equation*}
\begin{aligned}
 &(n+2)S\sum_{i,j}(\sum_kh_{ijk}\lambda_kD_k)^2+2n\sum_{i,j}(\sum_kh_{ijk}\lambda_k^2D_k)^2\le ntS^3\sum_k\lambda_k^2D_k^2,\\
 \end{aligned}
\end{equation*}
that is, 
\begin{equation}\label{eq:3.15}
\begin{aligned}
 &\sum_{i,j}(\sum_kh_{ijk}\lambda_kD_k)^2\le \dfrac{nt}{n+2}S^2D-\dfrac{2n}{n+2}\dfrac1S\sum_{i,j}(\sum_kh_{ijk}\lambda_k^2D_k)^2,\\
 \end{aligned}
\end{equation}
and
\begin{equation}\label{eq:3.16}
 \sum_i\left(\frac{n+2}{n}S+2\lambda_i^2\right)D_i^2\le tS^5|{\vec \varphi}_*|^2.
 \end{equation}
Since
\[
 \sum_i h_{iik}=\sum_i\lambda_i h_{iik}=\sum_i b_{ii}h_{iik}=0,
\]
we obtain from \eqref{eq:projection-definition}, for any $w_k$,
\begin{equation*}
\sum_kw_kD_k=S\sum _i{(\vec \varphi_{*})}_{i}\sum_kh_{iik}w_k.
\end{equation*}
Hence, we get 
\begin{equation*}
(\sum_kw_kD_k)^2\leq S^2|{\vec \varphi_{*}}|^2\sum_{i}(\sum_kh_{iik}w_k)^2\leq S^2|{\vec \varphi_{*}}|^2\sum_{i,j}(\sum_kh_{ijk}w_k)^2.
\end{equation*}
Taking $w_k=\lambda_k^2D_k$, we have
\begin{equation}\label{eq:3.17}
D^2=(\sum_k\lambda_k^2D_k^2)^2\leq S^2|{\vec \varphi_{*}}|^2\sum_{i,j}(\sum_kh_{ijk}\lambda_k^2D_k)^2.
\end{equation}
According to (\ref{eq:3.15}) and (\ref{eq:3.17}), we obtain
\begin{equation}\label{eq:3.18}
\begin{aligned}
 &\sum_{i,j}(\sum_kh_{ijk}\lambda_kD_k)^2\le \dfrac{nt}{n+2}S^2D-\dfrac{2n}{n+2}\dfrac{D^2}{S^3|{\vec \varphi_{*}}|^2},\\
 \end{aligned}
\end{equation}
Thus, we conclude, from  the definition of $\mathcal R_u$, (\ref{eq:3.18}) and 
$\frac{56}{9}+\frac{16u}{3}-\frac{16n}{3(n+2)}
 =\frac89+\frac{16u}{3}+\frac{32}{3(n+2)}$,
\begin{equation*}
\begin{aligned}
 &\mathcal R_u\leq -(\dfrac{8}{9}+\dfrac{16}{3}u+\dfrac{32}{3(n+2)})\dfrac{D}{tS^2}+\frac{21}{5}\left(4+\frac{8u}{3}\right)\dfrac{\sum_lD_l^2}S\\
 &-\dfrac{32n}{3(n+2)}\dfrac{D^2}{t^2S^7|{\vec \varphi_{*}}|^2}
 -4(5+\frac{4u}{3})\sum_l\frac{D_l^2}{S+\frac{2n}{n+2}\lambda_l^2}.\\
\end{aligned}
\end{equation*}
\end{proof}

\begin{lemma}\label{lem:3.5}
For the above limiting quantities with $0<t\le1/3$, the following estimates hold.
For $4\leq n\leq 12$, we have 
\begin{equation}
\begin{aligned}
0&\leq \mathcal R_{\frac12}+\mathcal E_{\frac12}+\mathcal G_{\frac12}\\
&\leq  \left(-\frac{2}{t}+\frac{16}{3}-\frac{2}{n}\right)tSf+\biggl[-\frac{32}{15}-\frac{t}{n^2}
+\frac{\left(\frac{97}{15}+\frac{t}{n}\right)^2}{17+t}\biggr]tf_3^2\\
& +\bigl[(2+\frac{3}{8n}+\dfrac{3}{2n^2})t-1+\frac{1}{3n}\bigr]tS^3+ \dfrac32S^3t|{\vec \varphi_{*}}|^2.
\end{aligned}
\end{equation}
For $n\geq 4$, we have 
\begin{equation}\label{eq:3.20}
\begin{aligned}
0&\leq \mathcal R_{\frac12}+\mathcal E_{\frac12}+\mathcal G_{\frac12}\\
&\leq  \left(-\frac{2}{t}+\frac{16}{3}-\frac{2}{n}\right)tSf+\biggl[-\frac{32}{15}-\frac{t}{n^2}
+\frac{\left(\frac{97}{15}+\frac{t}{n}\right)^2}{17+t}\biggr]tf_3^2\\
& +\bigl[(2+\frac{3}{8n}+\dfrac{3}{2n^2})t-1+\frac{1}{3n}\bigr]tS^3+ \frac{12}5S^3t|{\vec \varphi_{*}}|^2
\end{aligned}
\end{equation}
and
\begin{equation}\label{eq:3.21}
\begin{aligned}
0&\leq \mathcal R_{\frac14}+\mathcal E_{\frac14}+\mathcal G_{\frac14}\\
&\leq   \bigl(-\frac{15}{8t}+\frac{14}{3}-\frac{1}{n}\bigr)tSf
 +\left[-\frac{28}{15}-\frac{t}{n^2}+ \dfrac{(\frac{181}{30}+\frac{t}{n})^2}{16+t}\right]tf_3^2\\
&+\bigl[(2+\frac{3}{8n}+\dfrac{3}{2n^2})t-1+\frac{2}{3n}\bigr]tS^3+\dfrac85S^3t|{\vec \varphi_{*}}|^2.
\end{aligned}
\end{equation}

\end{lemma}
\begin{proof}
If $|\vec\varphi_*|=0$, \eqref{r:D} gives $D_k=0$ for all $k$.
The definition of $\mathcal R_u$ then gives $\mathcal R_u=0$.
We first estimate $\mathcal R_u$ when $|\vec\varphi_*|>0$.
For $v>0$, and any $k$
\[
 \frac1v-2k+k^2v=\frac{(1-kv)^2}{v}\geq 0.
\]
Taking $v=\frac{n+2}{n}S+2\lambda_i^2$, 
we obtain
$$
\sum_i\biggl[\dfrac{1}{\frac{n+2}nS+2\lambda_i^2}-2k+k^2(\frac{n+2}nS+2\lambda_i^2)\biggr]\dfrac{(n+2)D_i^2}{n}\geq 0, \\
$$
that is, 
\begin{equation}\label{eq:3.22}
\begin{aligned}
& \sum_i\frac{(n+2)D_i^2}{(n+2)S+2n\lambda_i^2}
 \ge{}\left(\frac{2(n+2)k}{n}-\frac{(n+2)^2k^2S}{n^2}\right)\sum_iD_i^2-\frac{2(n+2)k^2}{n}D.
\end{aligned}
\end{equation}
For $u=1/2$, taking  $k=3n/[4(n+2)S]$, we infer from (\ref{eq:3.22})
\begin{equation}\label{eq:3.23}
 \sum_i\frac{(n+2)D_i^2}{(n+2)S+2n\lambda_i^2}\ge\frac{15}{16}\dfrac{\sum_iD_i^2}S-\frac{9n}{8(n+2)}\dfrac{D}{S^2}.
\end{equation}
For $u=1/4$, taking
$k=13n/[20(n+2)S]$, we have 
\begin{equation}\label{eq:3.24}
 \sum_i\frac{(n+2)D_i^2}{(n+2)S+2n\lambda_i^2}\ge\frac{351}{400}\dfrac{\sum_iD_i^2}S-\frac{169n}{200(n+2)}\dfrac{D}{S^2}.
\end{equation}
Hence, we obtain from the lemma \ref{lem:3.4} and the above formulas
\begin{equation}\label{eq:3.25}
\begin{aligned}
 &\mathcal R_{\frac12}\leq \frac{23}{20}\dfrac{\sum_lD_l^2}S+\dfrac{51n}{2(n+2)}\dfrac{D}{S^2} -(\dfrac{32}{9}+\dfrac{32}{3(n+2)})\dfrac{D}{tS^2}
 -\dfrac{32n}{3(n+2)}\dfrac{D^2}{t^2S^7|{\vec \varphi_{*}}|^2}\\
\end{aligned}
\end{equation}
and
\begin{equation}\label{eq:3.26}
\begin{aligned}
 &\mathcal R_{\frac14}\leq \frac{22}{25}\dfrac{\sum_lD_l^2}S+\dfrac{1352n}{75(n+2)}\dfrac{D}{S^2} -(\dfrac{20}{9}+\dfrac{32}{3(n+2)})\dfrac{D}{tS^2}
 -\dfrac{32n}{3(n+2)}\dfrac{D^2}{t^2S^7|{\vec \varphi_{*}}|^2}.\\
\end{aligned}
\end{equation}
Defining
$$
z=\dfrac{D}{tS^5|{\vec \varphi_{*}}|^2}, 
$$
from the lemma \ref{lem:2.3}, we get 
\[
 0\le z\le\frac12,\qquad
\sum_kD_k^2\le\frac{n}{n+2}t S^4|{\vec \varphi_{*}}|^2(1-2z).
\]
From (\ref{eq:3.25}) and (\ref{eq:3.26}), we infer 
\begin{equation*}
\begin{aligned}
 &\dfrac{\mathcal R_{\frac12}}{S^3t|{\vec \varphi_{*}}|^2}
 \leq \frac{n}{n+2}\left(\frac{23}{20}+\frac{116}{5}z-\dfrac{32}{3t}z^2\right)
 -(\frac{32}{9t}+\frac{32}{3(n+2)t})z,
\end{aligned}
\end{equation*}
and
\begin{equation*}
\begin{aligned}
 &\dfrac{\mathcal R_{\frac14}}{S^3t|{\vec \varphi_{*}}|^2}
 \leq \frac{n}{n+2}\left(\frac{22}{25}+\frac{244}{15}z-\dfrac{32}{3t}z^2\right)
 -(\frac{20}{9t}+\frac{32}{3(n+2)t})z.
 \end{aligned}
\end{equation*}
Since $t\leq \dfrac13$, we have
\begin{equation}
\begin{aligned}
 &\dfrac{\mathcal R_{\frac12}}{S^3t|{\vec \varphi_{*}}|^2}
 \leq \frac{n}{n+2}\left(\frac{23}{20}+\frac{188}{15}z-32z^2\right)
 -\frac{160z}{3(n+2)},\\
\end{aligned}
\end{equation}
and
\begin{equation}
\begin{aligned}
 &\dfrac{\mathcal R_{\frac14}}{S^3t|{\vec \varphi_{*}}|^2}
 \leq \frac{n}{n+2}\left(\frac{22}{25}+\frac{48}{5}z-32z^2\right)
 -\frac{136z}{3(n+2)}.
\end{aligned}
\end{equation}
According to the identities
\begin{align*}
 \frac{23}{20}+\frac{188}{15}z-32z^2
 &=\frac{12}{5}-31\left(z-\frac15\right)^2
       -\left(z-\frac1{15}\right)^2-\frac1{180},\\
 \frac{22}{25}+\frac{48}{5}z-32z^2
 &=\frac85-32\left(z-\frac3{20}\right)^2
\end{align*}
and $n/(n+2)\le1$, we obtain
\begin{equation}
\mathcal R_{\frac12}\le\dfrac{12}5S^3t|{\vec \varphi_{*}}|^2, \quad 
\mathcal R_{\frac14}\le\dfrac{8}5S^3t|{\vec \varphi_{*}}|^2.
\end{equation}
 For $n\le12$ and  $u=1/2$, we can obtain a sharper upper bound of $\mathcal R_{1/2}$.  In fact, 
 \begin{equation}
\begin{aligned}
 &\dfrac{\mathcal R_{\frac12}}{S^3t|{\vec \varphi_{*}}|^2}
 \leq \frac{n}{n+2}\left(\frac{23}{20}+\frac{188}{15}z-32z^2\right)
 -\frac{160z}{3(n+2)},\\
\end{aligned}
\end{equation}
Since 
 $$
 \frac{n}{n+2}\left(\frac{23}{20}+\frac{188}{15}z-32z^2\right)
 -\frac{160z}{3(n+2)}
 $$
 is an increasing function of $n$
for  $0\le z\le1/2$, we have 
\[
 \dfrac{\mathcal R_{\frac12}}{S^3t|{\vec \varphi_{*}}|^2}
 \le\frac{69}{70}+\frac{104}{15}z-\frac{192}{7}z^2
 \le1+7z-\frac{49}{2}z^2\le\frac32.
\]
We have therefore proved
\[
 \mathcal R_{\frac12}\le
 \begin{cases}(3/2)S^3t|{\vec \varphi_{*}}|^2,&n\le12,\\(12/5)S^3t|{\vec \varphi_{*}}|^2,&n\ge4,\end{cases}
 \qquad \mathcal R_{\frac14}\le(8/5)S^3t|{\vec \varphi_{*}}|^2.
\]
These estimates also hold when $|\vec \varphi_*|=0$, as observed above.\\
For $u=\frac12$, we have
\begin{equation}\label{eq:3.31}
\begin{aligned}
 &\left(\frac{14}{5}\bigl(4+\frac{8u}{3}\bigr)-4u\right)\dfrac{f_3}SC  -\dfrac{15+4u}t \frac{C^2}{S^2}  
 -\dfrac{1}{S^2}(C-\frac{tS}{n}f_3)^2\\
 &=\frac{194}{15}\dfrac{f_3}SC -\dfrac{17}t \frac{C^2}{S^2}  
 -\dfrac{1}{S^2}(C-\frac{tS}{n}f_3)^2\\
 &= -\frac{17+t}{t}
 \biggl(\dfrac{C}{S}-\frac{t(\frac{97}{15}+\frac{t}{n})}{17+t} f_3\biggr)^2
 +t \biggl[\dfrac{(\frac{97}{15}+\frac{t}{n})^2}{17+t}-\dfrac{t}{n^2}\biggr] f_3^2\\
 &\leq t \biggl[\dfrac{(\frac{97}{15}+\frac{t}{n})^2}{17+t}-\dfrac{t}{n^2}\biggr] f_3^2.
\end{aligned}
\end{equation}
Thus, from (\ref{eq:3.31}) and definition of $\mathcal E_u$, we obtain
\begin{equation}\label{eq:3.32}
\begin{aligned}
\mathcal E_{\frac{1}{2}} &=\bigl(-\dfrac2t+\dfrac{16}3-\dfrac{2}n\bigr)tSf-\dfrac{32}{15}tf_3^2\\
&+ \frac{194}{15}\dfrac{f_3}SC -\dfrac{17}t \frac{C^2}{S^2}  
 -\dfrac{1}{S^2}(C-\frac{tS}{n}f_3)^2\\
&\leq  \left(-\frac{2}{t}+\frac{16}{3}-\frac{2}{n}\right)tSf
 +\left[-\frac{32}{15}-\frac{t}{n^2}
       +\frac{\left(\frac{97}{15}+\frac{t}{n}\right)^2}{17+t}\right]tf_3^2,\\
\end{aligned}
\end{equation}
\begin{equation}\label{eq:3.33}
\begin{aligned}
 \mathcal G_{\frac12}
 &=\bigl[(2+\frac{3}{8n}+\dfrac{3}{2n^2})t-1+\frac{1}{3n}\bigr]tS^3.
\end{aligned}
\end{equation}
For $u=\frac14$, we have
\begin{equation}\label{eq:3.34}
\begin{aligned}
 &\left(\frac{14}{5}\bigl(4+\frac{8u}{3}\bigr)-4u\right)\dfrac{f_3}SC  -\dfrac{15+4u}t \frac{C^2}{S^2}  
 -\dfrac{1}{S^2}(C-\frac{tS}{n}f_3)^2\\
 &=\frac{181}{15}\dfrac{f_3}SC -\dfrac{16}t \frac{C^2}{S^2}  
 -\dfrac{1}{S^2}(C-\frac{tS}{n}f_3)^2\\
 &= -\frac{16+t}{t}
 \biggl(\dfrac{C}{S}-\frac{t(\frac{181}{30}+\frac{t}{n})}{16+t} f_3\biggr)^2
 +t \biggl[\dfrac{(\frac{181}{30}+\frac{t}{n})^2}{16+t}-\dfrac{t}{n^2}\biggr] f_3^2\\
 &\leq t \biggl[\dfrac{(\frac{181}{30}+\frac{t}{n})^2}{16+t}-\dfrac{t}{n^2}\biggr] f_3^2.
\end{aligned}
\end{equation}
Thus, from (\ref{eq:3.34}), we obtain
\begin{equation}\label{eq:3.35}
\begin{aligned}
\mathcal E_{\frac{1}{4}}
 &=\bigl(-\frac{15}{8t}+\frac{14}{3}-\frac{1}{n}\bigr)tSf-\dfrac{28}{15}tf_3^2\\
&+\frac{181}{15}\dfrac{f_3}SC -\dfrac{16}t \frac{C^2}{S^2}  
 -\dfrac{1}{S^2}(C-\frac{tS}{n}f_3)^2\\
&\leq  \bigl(-\frac{15}{8t}+\frac{14}{3}-\frac{1}{n}\bigr)tSf
 +\left[-\frac{28}{15}-\frac{t}{n^2}
      + \dfrac{(\frac{181}{30}+\frac{t}{n})^2}{16+t}\right]tf_3^2,\\
\end{aligned}
\end{equation}
\begin{equation}\label{eq:3.36}
\begin{aligned}
 \mathcal G_{\frac14}&=\bigl[(2+\frac{3}{8n}+\dfrac{3}{2n^2})t-1+\frac{2}{3n}\bigr]tS^3.
\end{aligned}
\end{equation}
From (\ref{eq:3.32}),   (\ref{eq:3.33}),  (\ref{eq:3.35}) and  (\ref{eq:3.36}) we obtain
\begin{equation}
\begin{aligned}
0&\leq \mathcal R_{\frac12}+\mathcal E_{\frac12}+\mathcal G_{\frac12}\\
&\leq   \mathcal R_{\frac12}+\left(-\frac{2}{t}+\frac{16}{3}-\frac{2}{n}\right)tSf
 +\left[-\frac{32}{15}-\frac{t}{n^2}+\frac{\left(\frac{97}{15}+\frac{t}{n}\right)^2}{17+t}\right]tf_3^2\\
 &+\bigl[(2+\frac{3}{8n}+\dfrac{3}{2n^2})t-1+\frac{1}{3n}\bigr]tS^3
\end{aligned}
\end{equation}
For $4\leq n\leq 12$, we have 
\begin{equation}
\begin{aligned}
0&\leq \mathcal R_{\frac12}+\mathcal E_{\frac12}+\mathcal G_{\frac12}\\
&\leq  \left(-\frac{2}{t}+\frac{16}{3}-\frac{2}{n}\right)tSf+\biggl[-\frac{32}{15}-\frac{t}{n^2}
+\frac{\left(\frac{97}{15}+\frac{t}{n}\right)^2}{17+t}\biggr]tf_3^2\\
& +\bigl[(2+\frac{3}{8n}+\dfrac{3}{2n^2})t-1+\frac{1}{3n}\bigr]tS^3+ \dfrac32S^3t|{\vec \varphi_{*}}|^2
\end{aligned}
\end{equation}
and for $n\geq 4$, we have 
\begin{equation}
\begin{aligned}
0&\leq \mathcal R_{\frac12}+\mathcal E_{\frac12}+\mathcal G_{\frac12}\\
&\leq  \left(-\frac{2}{t}+\frac{16}{3}-\frac{2}{n}\right)tSf+\biggl[-\frac{32}{15}-\frac{t}{n^2}
+\frac{\left(\frac{97}{15}+\frac{t}{n}\right)^2}{17+t}\biggr]tf_3^2\\
& +\bigl[(2+\frac{3}{8n}+\dfrac{3}{2n^2})t-1+\frac{1}{3n}\bigr]tS^3+ \frac{12}5S^3t|{\vec \varphi_{*}}|^2
\end{aligned}
\end{equation}
and
\begin{equation}
\begin{aligned}
0&\leq \mathcal R_{\frac14}+\mathcal E_{\frac14}+\mathcal G_{\frac14}\\
&\leq   \bigl(-\frac{15}{8t}+\frac{14}{3}-\frac{1}{n}\bigr)tSf
 +\left[-\frac{28}{15}-\frac{t}{n^2}+ \dfrac{(\frac{181}{30}+\frac{t}{n})^2}{16+t}\right]tf_3^2\\
&+\bigl[(2+\frac{3}{8n}+\dfrac{3}{2n^2})t-1+\frac{2}{3n}\bigr]tS^3+\dfrac85S^3t|{\vec \varphi_{*}}|^2.
\end{aligned}
\end{equation}
\end{proof}

\begin{theorem}\label{thm:3.6}
Under assumptions of the theorem 1.1, suppose that $n<S\le n +\frac12n$.
For the limiting quantities obtained in \eqref{eq:3.8}, the following estimates hold.\\
if $n\le12$ or $|{\vec \varphi_{*}}|^2\le\frac{5}{9}|{\vec \varphi_{}}|^2=\frac{5}{9}\frac{f}{S^2}$, then
\begin{equation}\label{eq:necessary-half}
 0\le\frac{\dfrac{2f}{S^2}+\dfrac{f_3^2}{S^3}
-1}{9}+\frac{7}{36n},
\end{equation}
if $n\ge13$, then
\begin{equation}\label{eq:necessary-quarter}
 0\le\left(\frac{2}{9}-\frac{1}{3n}\right)\dfrac{f}{S^2}
 +\left(\frac{29}{240}+\frac{1}{12n}\right)\dfrac{f_3^2}{S^3}
 -\frac{1}{9}+\frac{5}{18n}.
\end{equation}
\end{theorem}

\begin{proof}
Keeping $f/S^2$ and $f_3^2/S^3$ fixed at their limiting values, we consider
functions $\mathcal N_{\frac12}(s)$ and $\mathcal N_{\frac14}(s)$ of $s>0$ defined by
 \begin{align*}
 \mathcal N_{\frac{1}{2}}(s)={}&
 \left(-\frac{2}{s}+\frac{16}{3}-\frac{2}{n}\right)\dfrac{f}{S^2}
 +\left[-\frac{32}{15}-\frac{s}{n^2}
       +\frac{\left(\frac{97}{15}+\frac{s}{n}\right)^2}{17+s}\right]\dfrac{f_3^2}{S^3}\\
 &+\left(2+\frac{3}{8n}+\frac{3}{2n^2}\right)s-1+\frac{1}{3n},\\
 \mathcal N_{\frac{1}{4}}(s)={}&
 \left(-\frac{15}{8s}+\frac{14}{3}-\frac{1}{n}\right)\dfrac{f}{S^2}
 +\left[-\frac{28}{15}-\frac{s}{n^2}
       +\frac{\left(\frac{181}{30}+\frac{s}{n}\right)^2}{16+s}\right]\dfrac{f_3^2}{S^3}\\
 &+\left(2+\frac{3}{8n}+\frac{3}{2n^2}\right)s-1+\frac{2}{3n}.
\end{align*}
Since 
\[
 \frac{d}{ds}\left(-\frac{s}{n^2}
             +\frac{(a+s/n)^2}{b+s}\right)
 =-\frac{(a-b/n)^2}{(b+s)^2},
\]
we obtain
\begin{align*}
 \mathcal N'_{\frac{1}{2}}(s)
 &=\frac{2}{s^2}\dfrac{f}{S^2}+2+\frac{3}{8n}+\frac{3}{2n^2}
       -\frac{(\frac{97}{15}-\frac{17}{n})^2}{(17+s)^2}\dfrac{f_3^2}{S^3}>0,\\
 \mathcal N'_{\frac{1}{4}}(s)
 &=\frac{15}{8s^2}\dfrac{f}{S^2}+2+\frac{3}{8n}+\frac{3}{2n^2}
       -\frac{(\frac{181}{30}-\frac{16}{n})^2}{(16+s)^2}\dfrac{f_3^2}{S^3}>0.
\end{align*}
Indeed, $0<\frac{97}{15}-\frac{17}{n}<7$ and
$0<\frac{181}{30}-\frac{16}{n}<7$, while $\dfrac{f_3^2}{S^3}\le1$.
Each negative term in these derivatives has absolute value less than
one, whereas the positive constant term is at least two. \\
Thus both
$\mathcal N_{\frac12}(s)$ and  $\mathcal N_{\frac14}(s)$ are strictly increasing on $0<s\le1/3$.
From the lemma \ref{lem:3.5}, we have, for $4\leq n\leq 12$, 
\begin{equation}
0\leq \mathcal N_{\frac12}(t)+\frac{3}2|{\vec  \varphi_{*}}|^2\leq \mathcal N_{\frac12}({\frac13})+\frac{3}2|{\vec \varphi_{*}}|^2,
\end{equation}
for $n\geq 4$, 
\begin{equation}
0\leq \mathcal N_{\frac12}(t)+\frac{12}5|{\vec  \varphi_{*}}|^2\leq \mathcal N_{\frac12}({\frac13})+\frac{12}5|{\vec \varphi_{*}}|^2,
\end{equation}
\begin{equation}
0\leq \mathcal N_{\frac14}(t)+\frac{8}5|{\vec \varphi_{*}}|^2\leq \mathcal N_{\frac14}({\frac13})+\frac{8}5|{\vec \varphi_{*}}|^2,
\end{equation}
that is, 
for $4\leq n\leq 12$,
\begin{equation}\label{eq:3.46}
\begin{aligned}
0&\le{}\left(-\frac29-\frac{2}{3n}\right)\dfrac{f}{S^2}
 +\left(\frac{1}{52}\left(\frac{97}{15}+\frac{1}{3n}\right)^2-\frac{32}{45}-\frac{1}{9n^2}\right)\dfrac{f_3^2}{S^3}\\
 &-\frac19+\frac{11}{72n}+\frac{1}{6n^2}+\frac12|{\vec \varphi_{*}}|^2,
\end{aligned}
\end{equation}
for $n\geq 4$,
\begin{equation}\label{eq:3.47}
\begin{aligned}
0&\le{}\left(-\frac29-\frac{2}{3n}\right)\dfrac{f}{S^2}
 +\left(\frac{1}{52}\left(\frac{97}{15}+\frac{1}{3n}\right)^2-\frac{32}{45}-\frac{1}{9n^2}\right)\dfrac{f_3^2}{S^3}\\
 &-\frac19+\frac{11}{72n}+\frac{1}{6n^2}+\frac45|{\vec \varphi_{*}}|^2,
 \end{aligned}
 \end{equation}
\begin{equation}\label{eq:3.48}
\begin{aligned}
0&\le{}\left(-\frac{23}{72}-\frac1{3n}\right)\dfrac{f}{S^2}
 +\left(\frac{1}{49}\left(\frac{181}{30}+\frac{1}{3n}\right)^2-\frac{28}{45}-\frac{1}{9n^2}\right)\dfrac{f_3^2}{S^3}\\
  &-\frac19+\frac{19}{72n}+\frac{1}{6n^2}+\frac8{15}|{\vec \varphi_{*}}|^2.
\end{aligned}
\end{equation}
From 
$$ 
\frac{121}{1300}<\frac{3}{32}, \quad \frac{97}{1170}<\frac{1}{12}, \quad \frac{17}{156}>\frac{1}{12},
$$
and 
$$
\frac{5321}{44100}<\frac{29}{240}, \quad \frac{181}{2205}<\frac{1}{12},
$$
we infer
\begin{equation}\label{eq:3.49}
\begin{aligned}
 \frac{1}{52}\left(\frac{97}{15}+\frac{1}{3n}\right)^2
       -\frac{32}{45}-\frac{1}{9n^2}
 &=\frac{121}{1300}+\frac{97}{1170n}-\frac{17}{156n^2}
 <\frac{1}{9},\\
 \frac{1}{49}\left(\frac{181}{30}+\frac{1}{3n}\right)^2
       -\frac{28}{45}-\frac{1}{9n^2}
 &=\frac{5321}{44100}+\frac{181}{2205n}-\frac{16}{147n^2}
 <\frac{29}{240}+\frac{1}{12n}.
\end{aligned}
\end{equation}
For $ n\geq 4$, we have 
\begin{equation}\label{eq:3.50}
 -\frac{1}{9}+\frac{11}{72n}+\frac{1}{6n^2}
 \le-\frac{1}{9}+\frac{7}{36n}.
\end{equation}
 For $n\geq 13$,  we have 
\begin{equation}\label{eq:3.51}
 -\frac{1}{9}+\frac{19}{72n}+\frac{1}{6n^2}
 \le-\frac{1}{9}+\frac{5}{18n}.
\end{equation}
For  $n\le12$,  according to (\ref{eq:2.2}), we have 
$$
|{\vec \varphi_{*}}|^2\leq |{\vec \varphi_{}}|^2=\dfrac{f}{S^2}.
$$
Since 
$-\frac{2}{9}-\frac{2}{3n}+\frac{1}{2}\le\frac{2}{9}$,
(\ref{eq:3.46}), (\ref{eq:3.49}) and (\ref{eq:3.50}) yield
$$
0\leq \dfrac29\dfrac{f}{S^2}+\dfrac19\dfrac{f_3^2}{S^3}-\frac{1}{9}+\frac{7}{36n}.
$$
If $|{\vec \varphi_{*}}|^2\leq \dfrac59|{\vec \varphi_{}}|^2=\dfrac59\dfrac{f}{S^2}$, 
since 
\[
 \left(-\frac29-\frac{2}{3n}\right)\dfrac{f}{S^2}+\frac45|{\vec \varphi_{*}}|^2
 \le\left(\frac29-\frac{2}{3n}\right)\dfrac{f}{S^2}\le\frac29\dfrac{f}{S^2},
\]
(\ref{eq:3.47}), (\ref{eq:3.49}) and (\ref{eq:3.50}) yield
$$
0\leq \dfrac29\dfrac{f}{S^2}+\dfrac19\dfrac{f_3^2}{S^3}-\frac{1}{9}+\frac{7}{36n}.
$$
If $n\geq 13$, 

$$
|{\vec \varphi_{*}}|^2\leq |{\vec \varphi_{}}|^2=\dfrac{f}{S^2}.
$$
yields
\[
 \left(-\frac{23}{72}-\frac{1}{3n}\right)\dfrac{f}{S^2}+\frac{8}{15}|{\vec \varphi_{*}}|^2
 \le\left(\frac{2}{9}-\frac{1}{3n}\right)\dfrac{f}{S^2}.
\]
In view of (\ref{eq:3.48}), (\ref{eq:3.49}), (\ref{eq:3.51}) and the above inequality, we conclude
 \eqref{eq:necessary-quarter}.
\end{proof}

\section{Proof of Theorem 1.1}

\begin{proof}[Proof of Theorem 1.1]
Assume $n<S\leq n+\frac12n$, so $0<t\leq\frac13$.
We use the limiting quantities obtained in Section 3.
If 
\begin{equation*}
2Sf+(\dfrac{n-1}{n-2}+\dfrac19)f^2_3> (1-\dfrac2n)S^3,
\end{equation*}
the lemma 2.1 yields
 $|{\bf \varphi_{*}}|^2\leq \dfrac59|{\bf \varphi_{}}|^2=\dfrac59\dfrac{f}{S^2}$.
 Hence, if  $n\le12$ or the preceding strict inequality holds, we have,  from the theorem 3.6 and the first inequality in the lemma 2.1,  
 \begin{equation}\label{eq:final-contradiction}
 0\le\frac{\dfrac{2f}{S^2}+\dfrac{f_3^2}{S^3}-1}{9}+\frac{7}{36n}\leq \dfrac19(-\frac{f_3^2}{(n-2)S^3}-\frac2n+\frac7{4n})\leq -\dfrac1{36n}<0.
\end{equation}
This is impossible.\\
It remains to consider $n\ge13$ and 
\begin{equation}
2Sf+(\dfrac{n-1}{n-2}+\dfrac19)f^2_3\leq  (1-\dfrac2n)S^3.
\end{equation}
We use
\[
 \frac{29}{240}+\frac{1}{12n}
 \le\left(\frac{1}{9}-\frac{1}{15n}\right)
       \left(\frac{n-1}{n-2}+\frac{1}{9}\right),
\]
since the difference between the right and left sides is
\[
 \frac{(n-12)(17n-130)+48}{6480n(n-2)}>0\qquad(n\ge13).
\]
Equation~\eqref{eq:necessary-quarter} in Theorem~\ref{thm:3.6} gives, for $n\geq13$,
\begin{equation*}
 0\le\left(\frac{2}{9}-\frac{1}{3n}\right)\dfrac{f}{S^2}
 +\left(\frac{29}{240}+\frac{1}{12n}\right)\dfrac{f_3^2}{S^3}
 -\frac{1}{9}+\frac{5}{18n}.
\end{equation*}
From  $\frac{2}{9}-\frac{1}{3n}\le2(\frac{1}{9}-\frac{1}{15n})$,
we obtain, from the assumption of this case, 
\begin{equation}
\begin{aligned}
 0&\le\left(\frac{2}{9}-\frac{1}{3n}\right)\dfrac{f}{S^2}
 +\left(\frac{29}{240}+\frac{1}{12n}\right)\dfrac{f_3^2}{S^3}
 -\frac{1}{9}+\frac{5}{18n}\\
 &\leq (\frac{1}{9}-\frac{1}{15n})\biggl[2\dfrac{f}{S^2}+(\frac{n-1}{n-2}+\frac{1}{9})\dfrac{f_3^2}{S^3}\biggr] -\frac{1}{9}+\frac{5}{18n}\\
 &\leq  (\frac{1}{9}-\frac{1}{15n})(1-\frac2n)-\frac{1}{9}+\frac{5}{18n}=\frac{12-n}{90n^2}<0.
 \end{aligned}
\end{equation}
Both cases are impossible. Therefore $S>n+\frac{1}{2}n$.
\end{proof}

\end{document}